\documentclass{amsproc}
\usepackage{amsfonts}

\theoremstyle{plain}

\newtheorem{corollary}{Corollary}

\newtheorem{definition}{Definition}

\newtheorem{lemma}{Lemma}

\newtheorem{proposition}{Proposition}
\newtheorem{remark}{Remark}

\newtheorem{theorem}{Theorem}
\numberwithin{equation}{section}

\begin{document}
\title[$L^{2}$-singular dichotomy ]{An abstract proof of the $L^{2}$%
-singular dichotomy for orbital measures on Lie algebras and groups}
\author{Kathryn E. Hare}
\address{Dept. of Pure Mathematics\\
University of Waterloo\\
Waterloo, Ont., Canada N2L 3G1}
\email{kehare@uwaterloo.ca}
\thanks{This research is supported in part by NSERC\ 45597}
\author{Jimmy He}
\address{Dept. of Mathematics\\
Stanford University\\
Stanford, CA., USA, 94305-2125}
\email{jimmyhe@stanford.edu}
\thanks{This paper is in final form and no version of it will be submitted
for publication elsewhere.}

\begin{abstract}
Let $G$ be a compact, connected simple Lie group and $\mathfrak{g}$ its Lie
algebra. It is known that if $\mu $ is any $G$-invariant measure supported
on an adjoint orbit in $\mathfrak{g}$, then for each integer $k$, the $k$%
-fold convolution product of $\mu $ with itself is either singular or in $%
L^{2}$. This was originally proven by computations that depended on the Lie
type of $\mathfrak{g}$, as well as properties of the measure. In this note,
we observe that the validity of this dichotomy is a direct consequence of
the Duistermaat-Heckman theorem from symplectic geometry and that, in fact,
any convolution product of (even distinct) orbital measures is either
singular or in $L^{2+\varepsilon }$ for some $\varepsilon >0$. An abstract
transference result is given to show that the $L^{2}$-singular
dichotomy holds for certain of the $G$-invariant measures supported on
conjugacy classes in $G.$
\end{abstract}

\maketitle

\section{Introduction}

Let $G$ be a compact, connected simple Lie group and $\mathfrak{g}$ its Lie
algebra. A classical result due to Ragozin \cite{Ra} states that the
convolution product of dimension of $\mathfrak{g}$, non-trivial orbital
measures (meaning, the $G$-invariant probability measures supported on the
adjoint orbits in $\mathfrak{g)}$, is absolutely continuous with respect to
the Lebesgue measure on $\mathfrak{g}$. Equivalently, the convolution
product measure has a Radon-Nikodym derivative in $L^{1}$. In a series of
papers, one of the authors, with various co-authors, significantly improved
upon this result, ultimately determining for each orbital measure the
minimum integer $k$ such that its $k$-fold convolution product is absolutely
continuous. Furthermore, it was shown that a dichotomy holds: either the $k$%
-fold product is purely singular to Lebesgue measure or its Radon-Nikodym
derivative is in $L^{1}\cap L^{2}(\mathfrak{g)}$.  The proof involved
detailed case-by-case analysis for each Lie type and orbital measure, see 
\cite{MathZ} and \cite{HJY}. This dichotomy was also shown to hold for a $k$%
-fold product of any invariant probability measure supported on a conjugacy
class in the Lie group (\cite{Adv}), and was subsequently extended to
products of distinct orbital measures in the Lie algebra $su(n)$ in \cite{Wr}%
.

Previously, Ricci and Stein (\cite{RS}, \cite{RS1}) had given an abstract
argument to show that if the $k$-fold convolution product has Radon-Nikodym
derivative in $L^{1}$, then it also belongs to $L^{1+\varepsilon }$ for some 
$\varepsilon >0$. But there was no suggestion in their proof that $%
\varepsilon $ could be as large as $1$. In this note, we see that the fact
that a convolution product of (possibly distinct) orbital measures on $%
\mathfrak{g}$ is absolutely continuous if and only if it is in $L^{2}(%
\mathfrak{g)}$ can be deduced abstractly from the Duistermaat-Heckman
theorem from symplectic geometry. Moreover, if such a product is absolutely
continuous, then it is actually in $L^{2+\varepsilon }(\mathfrak{g})$ for
every $\varepsilon <2$rank$\mathfrak{g}/(\dim \mathfrak{g}-$rank$\mathfrak{g}%
).$

We also show how the Dooley-Wildeberger wrapping map \cite{DW} can be used
to transfer the dichotomy result to the class of invariant measures
supported on conjugacy classes in $G$ that are the image under the
exponential map of adjoint orbits of the same dimension. Other than for the
Lie groups $G=$ $SU(n)$ (where all orbital measures on the group have this
additional property), it remains open if all products of any (distinct)
orbital measures on the Lie groups satisfy the $L^{2}$-singular dichotomy.

Acknowledgement: We are grateful to A. Wright for providing us with remarks
from M. Vergne pointing out the connection with the Duistermaat-Heckman
theorem and thank S. Gupta for helpful conversations.

\section{Duistermaat-Heckman Theorem and the Dichotomy for Orbital Measures
on the Lie Algebra}

Let $G$ be a compact, connected simple Lie group with maximal torus $T$ and
let $\mathfrak{g}$ and $\mathfrak{t}$ be the corresponding Lie algebras.
Assume $X\in \mathfrak{g}$. By the adjoint orbit $O_{X}\subseteq \mathfrak{g}
$ we mean the orbit generated by $X$ under the adjoint action of the
associated compact Lie group $G$,%
\begin{equation*}
O_{X}=\{Ad(g)X:g\in G\}\text{.}
\end{equation*}%
As observed in \cite{DRW}, any adjoint orbit $O$ is a symplectic manifold,
under the identification with the co-adjoint orbit, with the 2-form given by 
$\Omega _{Z}([X,Z],[Y,Z])=(Z,[X,Y])$ for $Z\in O$ and $X,Y\in \mathfrak{g}$.
Here $[\cdot ,\cdot ]$ denotes the Lie bracket. This 2-form is invariant
under the $G$-action and with the moment map $\phi :O\rightarrow \mathfrak{g}
$ given by inclusion, $(O,\Omega ,\phi )$ is a Hamiltonian $G$-space.

The product of orbits, $O_{X}\times O_{Y},$ is also a Hamiltonian $G$-space
with $2$-form equal to the sum of the 2-forms on $O_{X}$ and $O_{Y}$ and
moment map $\phi (Z_{1},Z_{2})=Z_{1}+Z_{2}$ for $Z_{1}\in O_{X}$ and $%
Z_{2}\in O_{Y}$.

\begin{definition}
By the orbital measure $\mu _{X}$ on $\mathfrak{g}$ we mean the unique $G$%
-invariant probability measure, supported on $O_{X}$, given by%
\begin{equation*}
\int_{\mathfrak{g}}fd\mu _{X}=\int_{G}f(Ad(g)X)dg\text{ }
\end{equation*}%
for all continuous, compactly supported functions $f$ on $\mathfrak{g}$.
Here $dg$ is the Haar measure on $G$.
\end{definition}

When viewed on $O_{X}$, $\mu _{X}$ is the Liouville measure on $O_{X}$ (up
to normalization). Viewed on $\mathfrak{g,}$ the orbital measure is the
pushforward measure under the moment map and hence the Duistermaat-Heckman
measure. The product measure $\mu _{X}\times \mu _{Y}$ is the Liouville
measure on $O_{X}\times O_{Y}$ and the pushforward under the moment map is
the convolution of the orbital measures, $\mu _{X}\ast \mu _{Y}$.

The well known Duistermaat-Heckman theorem states the following.

\begin{theorem}
(c.f. \cite{DH}, \cite{ET}) Suppose $N$ is a compact Hamiltonian $T$-space
with $T$ a torus and proper moment map $\phi $. The pushforward of the
Liouville measure on $N$ is a polynomial on each connected component of the
regular values of $\phi ,$ of degree at most $\dim N/2-\dim T.$
\end{theorem}

We will fix a positive Weyl chamber and let $\mathfrak{t}^{0}$ denote the
interior of this set. Let $\Phi ^{+}$ be the set of positive roots and
denote by $\pi $ the $G$-invariant function given by%
\begin{equation*}
\pi (H)=\prod_{\alpha \in \Phi ^{+}}\alpha (H)\text{ for }H\in \mathfrak{t.}
\end{equation*}

The following change of variables formula is well known (in more generality).

\begin{proposition}
(c.f. \cite{DHV}) Let $M=O_{X_{1}}\times \cdot \cdot \cdot \times O_{X_{L}},$
where $X_{i}\in \mathfrak{g}$ and $\phi :M\rightarrow \mathfrak{g}$ is the
addition map. Denote by $\mu _{M}$ the Liouville measure on $M$, $\mu _{M}=$ 
$\mu _{X_{1}}\times \cdot \cdot \cdot \times \mu _{X_{L}}$. Assume that the
Duistermaat-Heckman measure on $\phi (M),$ $\mu _{X_{1}}\ast \cdot \cdot
\cdot \ast \mu _{X_{L}}$, is an absolutely continuous measure. Then the
manifold $N=\phi ^{-1}(\mathfrak{t}^{0})$ is a Hamiltonian $T$-space and if $%
\mu _{N}$ is its Liouville measure, there is a constant $c>0$ such that%
\begin{equation*}
\int_{M}f\circ \phi d\mu _{M}=c\int_{N}f\circ \phi \text{ }\left\vert \pi
\circ \phi \right\vert \text{ }d\mu _{N}
\end{equation*}%
for all $G$-invariant Borel functions $f$.
\end{proposition}

\begin{remark}
The assumption that $\mu _{X_{1}}\ast \cdot \cdot \cdot \ast \mu _{X_{L}}$
is absolutely continuous is equivalent to saying that $O_{X_{1}}+\cdot \cdot
\cdot +O_{X_{L}}=\phi (M)$ has positive Lebesgue measure in $\mathfrak{g}$,
see \cite{CJM} or \cite{Ra}.
\end{remark}

As $N$ is not compact we cannot apply the Duistermaat-Heckman theorem
directly. Rather, we will instead consider the compact symplectic orbifolds $%
N_{\varepsilon }$ formed by taking the symplectic cut in each root vector
direction, thereby obtaining nested, compact sets whose union over all $%
\varepsilon >0$ is $N$. An easy limiting argument shows that 
\begin{equation*}
\int_{N}f\circ \phi \text{ }\left\vert \pi \circ \phi \right\vert \text{ }%
d\mu _{N}=\lim_{\varepsilon \rightarrow 0}\int_{N_{\varepsilon }}f\circ \phi 
\text{ }\left\vert \pi \circ \phi \right\vert \text{ }d\mu _{N_{\varepsilon
}}.
\end{equation*}

Denote by $\nu _{M}$ the pushforward of $\mu _{M}$ under $\phi $ and by $\nu
_{N_{\varepsilon }}$ the pushforward of the Liouville measure $\mu
_{N_{\varepsilon }}$ on $N_{\varepsilon }$. As the Duistermaat-Heckman
theorem holds for orbifolds (\cite{DG}), we can apply it to $\nu
_{N_{\varepsilon }}$ to deduce that there are functions, $P_{\varepsilon },$
which are locally polynomial and of bounded degree, such that 
\begin{eqnarray*}
\int_{\mathfrak{g}}f(X)d\nu _{M}(X) &=&\int_{M}f\circ \phi d\mu
_{M}=c\lim_{\varepsilon }\int_{N_{\varepsilon }}f\circ \phi \text{ }%
\left\vert \pi \circ \phi \right\vert \text{ }d\mu _{N_{\varepsilon }} \\
&=&\lim_{\varepsilon }\int_{\mathfrak{t}^{0}}f(H)\left\vert \pi
(H)\right\vert d\nu _{N_{\varepsilon }}(H) \\
&=&\lim_{\varepsilon }\int_{\mathfrak{t}^{0}}f(H)\left\vert \pi
(H)\right\vert P_{\varepsilon }(H)dH.
\end{eqnarray*}

On the other hand, if $R$ denotes the Radon-Nikodym derivative of $\nu _{M}$%
, then the Weyl integration formula gives that for $G$-invariant functions $%
f $,%
\begin{equation*}
\int_{\mathfrak{g}}f(X)d\nu _{M}(X)=\int_{\mathfrak{g}}f(X)R(X)dX=\int_{%
\mathfrak{t}^{0}}f(H)R(H)\left\vert \pi (H)\right\vert ^{2}\text{ }dH.
\end{equation*}

Uniqueness of the Radon-Nikodym derivative implies that up to a
normalization constant, $R=\lim_{\varepsilon }P_{\varepsilon }/\left\vert
\pi \right\vert $ on $\mathfrak{t}^{0}$. We further note that $%
P=\lim_{\varepsilon }P_{\varepsilon }$ is compactly supported on $\mathfrak{t%
}$ as $R$ is supported on the compact set $\phi (M)$. Properties of the
moment map ensure that there are only finitely many connected components of
the set of regular elements of $\phi ,$ hence $P$ is bounded.

Using the Weyl integration formula again, we can compute the $L^{2}$ norm of 
$R$ obtaining%
\begin{equation*}
\int_{\mathfrak{g}}\left\vert R\right\vert ^{2}dX=\int_{\mathfrak{t}%
^{0}}\left\vert P(H)\right\vert ^{2}dH<\infty .
\end{equation*}
Thus the assumption that $\nu _{M}$ is absolutely continuous guarantees that
its Radon-Nikodym derivative is in $L^{1}\cap L^{2}(\mathfrak{g)}$. Since a
product of orbital measures is known to be either purely singular or
absolutely continuous (c.f. \cite{CJM} or \cite{Ra}) this gives an abstract
proof of the $L^{2}$-singular dichotomy:

\begin{corollary}
Let $X_{i}\in \mathfrak{g}$ for $i=1,...,L$. Then $\mu _{X_{1}}\ast \cdot
\cdot \cdot \ast \mu _{X_{L}}$ is either singular or its Radon-Nikodym
derivative is in $L^{1}\cap L^{2}(\mathfrak{g)}$.
\end{corollary}

\section{The Dichotomy for Orbital Measures on the Group}

It is natural to ask if there is a similar dichotomy result for orbital
measures on the compact Lie group, where by an orbital measure in this
setting we mean the invariant probability measure $\mu _{x}$, supported on
the conjugacy class $C_{x}$ generated by $x$. This measure integrates
according to the rule%
\begin{equation*}
\int_{G}fd\mu _{x}=\int_{G}f(Ad(g)x)dg\text{ for }f\in C(G)\text{.}
\end{equation*}%
In \cite{Adv} it was shown that the $L^{2}$-singular dichotomy does hold for
convolutions of a given orbital measure. This was proven by using
representation theory and the Peter-Weyl theorem to calculate the $L^{2}(G)$
norm for such measures. (In fact, the strategy of the earlier work on the $%
L^{2}$-singular dichotomy was to do these calculations first for orbital
measures on the group and then use the answers to study the analogous
problem on the Lie algebra.)

Applying this technique to convolutions of different orbital measures is
cumbersome, however, and has been done only for $G=SU(n)$ in \cite{Wr}.
Instead, in this note we will take the opposite approach and will see how to
use the Lie algebra dichotomy to deduce the dichotomy for certain
convolution products of (possibly distinct) orbital measures on the group.

This transference argument will rely upon the wrapping map $\Psi $
introduced in \cite{DW}. Given a measure $\mu $ compactly supported on $%
\mathfrak{g}$, we define a measure $\Psi (\mu )$ on $G$ by 
\begin{equation*}
\int_{G}fd\Psi (\mu )=\int_{\mathfrak{g}}jf\circ \exp d\mu \text{ for }f\in
C(G)
\end{equation*}%
where $j$ is the analytic square root of the determinant of the exponential
map satisfying $j(0)=1$; 
\begin{equation*}
j(H)=\prod_{\alpha \in \Phi ^{+}}\frac{\sin (\alpha (H)/2)}{\alpha (H)/2}%
\text{ for }H\in \mathfrak{t.}
\end{equation*}%
It was shown in \cite{DW} that $\Psi (\mu \ast \nu )=\Psi (\mu )\ast \Psi
(\nu )$ and it is easy to see from the definition that if $\mu \in L^{1}(%
\mathfrak{g)}$ with Radon-Nikodym derivative $F$, then $\Psi (\mu )\in
L^{1}(G)$ and has Radon Nikodym derivative $\Psi (F)$. In the lemma below we
list some further properties of $\Psi $.

Notation: Let $\Gamma =\{\exp ^{-1}(e)\}\bigcap \mathfrak{t}$ ($e$ being
the identity in $G)$.

\begin{lemma}
\label{prelim}(1) If $f\in C^{\infty }(\mathfrak{g)}$ is $G$-invariant, then 
$\Psi (jf)$ is $G$-invariant, and 
\begin{equation*}
\Psi (jf)(\exp H)=\sum_{\alpha \in \Gamma }f(H+\gamma )\text{ for all }H\in 
\mathfrak{t.}
\end{equation*}

(2) For any $H\in \mathfrak{t}$, $\Psi (\mu _{H})=j(H)\mu _{\exp H}$. More
generally, if $X_{i}\in \mathfrak{t}$ and $x_{i}=\exp X_{i}$, then $\Psi
(\mu _{X_{1}}\ast \cdot \cdot \cdot \ast \mu
_{X_{L}})=\prod_{i=1}^{L}j(X_{i})\mu _{x_{1}}\ast \cdot \cdot \cdot \ast \mu
_{x_{L}}$.
\end{lemma}

\begin{proof}
(1) is shown in \cite{DW}. To prove (2) we simply note that as $j$ is $G$%
-invariant, the definition of the orbital measure implies that for any $G$%
-invariant, continuous function $f$ on $G$, we have%
\begin{eqnarray*}
\int_{G}fd\Psi (\mu _{H}) &=&\int_{\mathfrak{g}}j(X)f\circ \exp (X)d\mu
_{H}=\int_{G}j(H)f\circ \exp (Ad(g)H)dg\text{ } \\
&=&j(H)\int_{G}f(Ad(g)\exp H)dg=j(H)\int_{G}fd\mu _{\exp H}\text{.}
\end{eqnarray*}
\end{proof}

We will deduce the dichotomy property from the following result that may be
of independent interest.

\begin{proposition}
\label{group}Let $K$ be a compact subset of $\mathfrak{g}$ and $1\leq
p<\infty $. There is a constant $C=C(K,p)$ such that 
\begin{equation*}
\left\Vert \Psi (jf)\right\Vert _{L^{p}(G)}\leq C\left\Vert
j^{2/p}f\right\Vert _{L^{p}(\mathfrak{g)}}
\end{equation*}%
for all $G$-invariant, Borel functions supported on $K$.
\end{proposition}

\begin{proof}
As in \cite{DW}, we will normalize the Haar measures $dg$ on $G$ and $dt$ on 
$T$ to have norm one and normalize the Lebegue measures $dX$ on $\mathfrak{g}
$ and $dH$ on $\mathfrak{t}$ so that if $U$ is a neighbourhood of $0$ in $%
\mathfrak{g}$ (respectively $\mathfrak{t)}$ on which the exponential map is
injective, then for continuous $f$ on $G$ (or $T)$ we have 
\begin{equation*}
\int_{U}f\circ \exp X\left\vert j(X)\right\vert ^{2}dX=\int_{\exp U}f(g)dg
\end{equation*}%
(respectively, $\int_{U}f\circ \exp HdH=\int_{\exp U}f(t)dt$).

Let $|\Delta (\exp H)|^{2}=\left\vert j\pi (H)\right\vert ^{2}$. The Weyl
integration formula gives%
\begin{equation*}
\left\Vert \Psi (jf)\right\Vert _{L^{p}(G)}^{p}=\frac{1}{\left\vert
W\right\vert }\int_{T}\left\vert \Delta (t)\right\vert ^{2}\left\vert \Psi
(jf(t))\right\vert ^{p}dt.
\end{equation*}

Let $t_{\Gamma }$ be a fundamental domain for $\Gamma $ in $\mathfrak{t}$.
Our choice of normalization and Lemma \ref{prelim}(1) shows that 
\begin{eqnarray*}
\int_{T}\left\vert \Delta \right\vert ^{2}\left\vert \Psi (jf)\right\vert
^{p}dt &=&\int_{t_{\Gamma }}\left\vert \Delta (\exp H)\right\vert
^{2}\left\vert \Psi (jf)\circ \exp H\right\vert ^{p}dH \\
&=&\int_{t_{\Gamma }}\left\vert \Delta (\exp H)\right\vert ^{2}|\sum_{\gamma
\in \Gamma }f(H+\gamma )|\left\vert \Psi (jf)\circ \exp H\right\vert ^{p-1}dH
\\
&\leq &\sum_{\gamma }\int_{t_{\Gamma }}\left\vert \Delta (\exp H)\right\vert
^{2}\left\vert f(H+\gamma )\right\vert \left\vert \Psi (jf)\circ \exp
H\right\vert ^{p-1}dH.
\end{eqnarray*}%
Since $\exp (H+\gamma )=\exp H$ for $\gamma \in \Gamma $ and $f$ is
supported on the compact set $K,$ this sum is equal to 
\begin{equation*}
\int_{\mathfrak{t}\cap K}\left\vert \Delta (\exp H)\right\vert
^{2}\left\vert f(H)\right\vert \left\vert \Psi (jf)\circ \exp H\right\vert
^{p-1}dH.
\end{equation*}%
Applying Holder's inequality with conjugate indices $p,p^{\prime }$, it
follows that this integral is bounded by%
\begin{equation*}
\left( \int_{\mathfrak{t}\cap K}\left\vert \Delta (\exp H)\right\vert
^{2}\left\vert f(H)\right\vert ^{p}dH\right) ^{1/p}\left( \int_{\mathfrak{t}%
\cap K}\left\vert \Delta (\exp H)\right\vert ^{2}\left\vert \Psi (jf)\circ
\exp H\right\vert ^{p}dH\right) ^{1/p^{\prime }}
\end{equation*}%
\begin{equation*}
:=I_{1}\cdot I_{2}.
\end{equation*}%
Another application of the Weyl integration formula gives that 
\begin{equation*}
I_{1}=\left( \int_{\mathfrak{g}}|j(X)|^{2}|f(X)|^{p}dX\right)
^{1/p}=\left\Vert j^{2/p}f\right\Vert _{L^{p}(\mathfrak{g)}}
\end{equation*}%
while Lemma \ref{prelim}(1) gives 
\begin{equation*}
I_{2}=\left( \int_{\mathfrak{t\cap }K}\left\vert \Delta (\exp H)\right\vert
^{2}|\sum_{\gamma \in \Gamma }f(H+\gamma )|^{p}dH\right) ^{1/p^{\prime }}.
\end{equation*}

Since $\Gamma $ is a discrete set, there are only finitely many $\gamma \in
\Gamma $ such that $H+\gamma \in K$ for some $H\in K$. Hence for a constant $%
C$ (which may change from one occurrence to another) 
\begin{eqnarray*}
I_{2} &\leq &\left( C\int_{\mathfrak{t}}\left\vert \Delta (\exp
H)\right\vert ^{2}\left\vert f(H+\gamma )\right\vert ^{p}dH\right)
^{1/p^{\prime }} \\
&\leq &\left( C\int_{\mathfrak{g}}\left\vert j(X)\right\vert ^{2}\left\vert
f(X)\right\vert ^{p}dH\right) ^{1/p^{\prime }}=C\left\Vert
j^{2/p}f\right\Vert _{L^{p}(\mathfrak{g)}}^{p/p^{\prime }}.
\end{eqnarray*}%
Combining the bounds on $I_{1},I_{2}$ gives the desired result.
\end{proof}

\begin{corollary}
\label{main}Assume that $x_{i}=\exp X_{i}$ and $\dim C_{x_{i}}=\dim
O_{X_{i}} $ for $i=1,...,L$. Then $\mu _{x_{1}}\ast \cdot \cdot \cdot \ast
\mu _{x_{L}} $ is either singular with respect to Haar measure on $G$ or its
Radon-Nikodym derivative is in $L^{2}(G\mathfrak{)}$.
\end{corollary}

\begin{proof}
It was seen in \cite{CJM} that under the assumption that $x_{i}=\exp X_{i}$
and $\dim C_{x_{i}}=\dim O_{X_{i}},$ $\mu _{x_{1}}\ast \cdot \cdot \cdot
\ast \mu _{x_{L}}$ is absolutely continuous (on $G)$ if and only if $\mu
_{X_{1}}\ast \cdot \cdot \cdot \ast \mu _{X_{L}}\mathfrak{\ }$is absolutely
continuous (on $\mathfrak{g}$). In the previous section, under the latter
assumption we saw that the Radon-Nikodym derivative of $\mu _{X_{1}}\ast
\cdot \cdot \cdot \ast \mu _{X_{L}}$ is a $G$-invariant function which on $%
\mathfrak{t}^{0}$ has the form $P/|\pi |\mathfrak{,}$ where $P$ is a
compactly supported, bounded function.

The hypothesis that $\dim O_{X_{i}}=\dim C_{x_{i}}$ guarantees that $\sin
\alpha (X_{i})=0$ only if $\alpha (X_{i})=0$ for $\alpha \in \Phi $. Hence $%
j(X_{i})\neq 0$ for all $i$. Lemma \ref{prelim}(2) implies that $\mu
_{x_{1}}\ast \cdot \cdot \cdot \ast \mu _{x_{L}}$ has Radon-Nikodym
derivative $C\Psi (P/|\pi |)$ for $C=\prod_{i}j(X_{i})^{-1}$. Hence it
suffices to show that $\Psi (P/|\pi |)=\Psi (jP/(\Delta \circ \exp ))\in
L^{2}(G)$. By the previous theorem 
\begin{equation*}
\left\Vert \Psi \left( \frac{jP}{\Delta \circ \exp }\right) \right\Vert
_{L^{2}(G)}^{2}\leq C\left\Vert \frac{jP}{\Delta \circ \exp }\right\Vert
_{L^{2}(\mathfrak{g)}}^{2}\leq C\int_{\mathfrak{g}}\left\vert \frac{P}{\pi }%
\right\vert ^{2}\leq C\int_{\mathfrak{t}}\left\vert P\right\vert ^{2},
\end{equation*}%
and the latter integral is finite as $P$ is compactly supported and bounded.
\end{proof}

\section{The $L^{2+\protect\varepsilon }$ argument}

In fact, we can prove that if a convolution of orbital measures on $%
\mathfrak{g}$ is absolutely continuous, then its Radon-Nikodym derivative is
actually in $L^{p}(\mathfrak{g})$ for some $p>2$.

\begin{proposition}
If $\mu =\mu _{X_{1}}\ast \cdot \cdot \cdot \ast \mu _{X_{L}}$ is absolutely
continuous with respect to Lebesgue measure on $\mathfrak{g}$, then its
Radon Nikodym derivative belongs to $L^{2+\varepsilon }(\mathfrak{g)}$ for
every $\varepsilon <2$rank$\mathfrak{g}/(\dim \mathfrak{g}-$rank$\mathfrak{g}%
)$.
\end{proposition}

\begin{proof}
This will follow from a result of Stanton and Tomas \cite{ST} that states $%
\int_{T}\left\vert \Delta \right\vert ^{-\varepsilon }$ is finite for such $%
\varepsilon $.

From Section 2 we know that the Radon-Nikodym derivative of $\mu $ is equal
to $P/|\pi |$ on $\mathfrak{t}^{0}$, for $P$ bounded and compactly
supported. Let $t_{\Gamma }$ be a fundamental domain for $\Gamma $ in $%
\mathfrak{t}$ that is precompact and $K$ be the compact support of $P.$
Since $\left\vert \pi \right\vert \geq \left\vert \Delta \circ \exp
\right\vert $ we have 
\begin{eqnarray*}
\int_{\mathfrak{g}}\left\vert P/\pi \right\vert ^{2+\varepsilon } &=&\int_{%
\mathfrak{t}^{0}}\left\vert P\right\vert ^{2+\varepsilon }\left\vert \pi
\right\vert ^{-\varepsilon }\leq \int_{\mathfrak{t}}\frac{\left\vert
P(H)\right\vert ^{2+\varepsilon }}{\left\vert \Delta (\exp H)\right\vert
^{\varepsilon }}dH=\sum_{\gamma \in \Gamma }\int_{\mathfrak{t}_{\Gamma }}%
\frac{\left\vert P(H+\gamma )\right\vert ^{2+\varepsilon }}{\left\vert
\Delta (\exp H)\right\vert ^{\varepsilon }}dH \\
&=&\sum_{\gamma \in \Gamma }\int_{K\cap (\mathfrak{t}_{\Gamma }-\gamma )}%
\frac{\left\vert P(H)\right\vert ^{2+\varepsilon }}{\left\vert \Delta (\exp
H)\right\vert ^{\varepsilon }}dH.
\end{eqnarray*}

As $\Gamma $ is discrete there can only be finitely many $\gamma \in \Gamma $
such that $K\cap (\mathfrak{t}_{\Gamma }-\gamma )$ is not empty. Thus
similar reasoning to the proof of Prop. \ref{group}, coupled with the
Stanton-Tomas result, shows that the sum above is finite and hence $\mu \in
L^{2+\varepsilon }(\mathfrak{g)}$.
\end{proof}

\end{document}